\documentclass[12pt]{amsart}
\usepackage{geometry}
\usepackage{mathtools}

\usepackage{amssymb,amsmath,amsthm}
\usepackage[english]{babel}
\usepackage{hyperref}

\newcommand{\mf}[1]{\mathfrak #1}

\newcommand{\dkp}{\dim_{k(\mf p)}}
\newcommand{\dkq}{\dim_{k(\mf q)}}
\newcommand{\ul}{\underline}

\renewcommand{\frq}[1]{{#1}^{[p^e]}}
\newcommand{\red}[1]{{#1}_{\text{red}}}

\DeclareMathOperator{\hght}{ht}
\DeclareMathOperator{\length}{\lambda}

\DeclareMathOperator{\eh}{e}
\DeclareMathOperator{\ehk}{e_{HK}}

\DeclareMathOperator{\Max}{Max}
\DeclareMathOperator{\gr}{Gr}

\DeclareMathOperator{\Spec}{Spec}
\DeclareMathOperator{\coker}{coker}
\DeclareMathOperator{\Trace}{Tr}
\DeclareMathOperator{\disc}{D}
\DeclareMathOperator{\chr}{char}
\DeclareMathOperator{\rsig}{s_{rat}} 
\DeclareMathOperator{\csig}{s_{rel}} 

\newtheorem{theorem}{Theorem}
\newtheorem{lemma}[theorem]{Lemma}
\newtheorem{proposition}[theorem]{Proposition}
\newtheorem{corollary}[theorem]{Corollary}
\newtheorem{claim}{Claim}
\theoremstyle{definition}
\newtheorem{definition}[theorem]{Definition}

\theoremstyle{remark}
\newtheorem{remark}[theorem]{Remark}
\newtheorem{example}[theorem]{Example}

\numberwithin{theorem}{section}
\numberwithin{equation}{section}

\begin{document}

\title
{On semicontinuity of multiplicities in families}

\author{Ilya Smirnov}
\address{Department of Mathematics, Stockholm University, SE-106 91, Stockholm, Sweden}
\email{smirnov@math.su.se}

\date{\today}

\begin{abstract}
This paper investigates the behavior of Hilbert--Samuel multiplicity and Hilbert--Kunz multiplicity in families of ideals. We show that Hilbert--Samuel multiplicity is upper semicontinuous and that Hilbert--Kunz multiplicity is upper semicontinuous in families of finite type. 
As a consequence, F-rational signature, an invariant defined by Hochster and Yao
as the infimum of relative Hilbert--Kunz multiplicities, is, in fact, a minimum.
This gives a different proof for its main property: F-rational signature is positive if and only if the ring is F-rational.
The tools developed in this paper can be also applied to families over $\mathbb Z$ and yield a solution to Claudia Miller's question on reduction mod $p$ of Hilbert--Kunz function. 
\end{abstract}

\maketitle

\section{Introduction}
Hilbert--Kunz multiplicity is a multiplicity theory native to positive characteristic. 
Its definition mimics the definition of Hilbert--Samuel multiplicity 
but replaces regular powers $I^n$ with Frobenius powers $\frq{I} = \{x^{p^e} \mid x \in I\}$.
The Hilbert--Kunz multiplicity of an $\mf m$-primary ideal $I$
of a local ring $(R, \mf m)$ is the limit
\[
\ehk(I) = \lim_{e \to \infty} \frac{\length_R (R/\frq{I})}{p^{e\dim R}}.
\]
It is not easy to see that the above limit exists. Existence was shown by Monsky, 
who introduced the concept in \cite{Monsky} as a continuation of earlier work of Kunz 
\cite{Kunz1, Kunz2}.

Hilbert--Kunz multiplicity is very hard to calculate and Paul Monsky was a driving force behind most of the known examples.
Several interesting families appear in literature:
plane cubics (\cite{MonskyCubic,MonskyCubic2,BuchweitzChen,Pardue}),
quadrics in characteristic two (\cite{MonskyQL,MonskyQP}),
and another family in \cite{MonskyFamily}.
The most famous of these families 
is the one appearing in \cite{MonskyQP}.

\begin{example}\label{ex Monsky}
Let $K$ be an algebraically closed field of characteristic $2$. 
For $\alpha \in K$ let $R_\alpha = K[x,y,z]/(z^4 +xyz^2 + (x^3 + y^3)z + \alpha x^2y^2)$
localized at $(x,y,z)$.
Then
\begin{enumerate}
\item $\ehk (R_\alpha) = 3 + \frac 12$, if $\alpha = 0$,
\item $\ehk (R_\alpha) = 3 + 4^{-m}$, 
if $\alpha \neq 0$ is algebraic over $\mathbb Z/2\mathbb Z$, where 
$m = [\mathbb Z/2\mathbb Z(\lambda): \mathbb Z/2\mathbb Z]$ for $\lambda$ such that $\alpha = \lambda^2 + \lambda$
\item $\ehk (R_\alpha) = 3$ if $\alpha$ is transcendental over $\mathbb Z/2\mathbb Z$.
\end{enumerate}
\end{example} 

Monsky's computations were later used by him and Brenner 
to give in \cite{BrennerMonsky} a counter-example to an outstanding problem in the field: localization of tight closure,
the problem originating already from the foundational treatise of Hochster and Huneke \cite{HochsterHuneke1}.
For this result, it is better to think about the example as a family of rings 
parametrized by $\Spec K[t]$
and the necessary phenomenon is the jump in the values between the generic fiber, 
corresponding to transcendental values, and special fibers, corresponding to algebraic values.

Another consequence of Monsky's example was found by the author in \cite{equi},
where it was shown that Hilbert--Kunz multiplicity
takes infinitely many values as a function on 
\[
\Spec K[x,y,z, t]/(z^4 +xyz^2 + (x^3 + y^3)z + tx^2y^2)
\]
by developing a technique of lifting this phenomenon from special fibers
to the corresponding maximal ideals $\mf m_\alpha = (x,y,z, t - \alpha)$.

Semicontinuity in Hilbert--Kunz theory was already studied by Kunz, who showed in \cite{Kunz2} upper semicontinuity
of individual terms of the sequence (also, see \cite{ShepherdBarron}), but the real momentum was given by Enescu and Shimomoto 
in \cite{EnescuShimomoto}, where they investigated both semicontinuity of Hilbert--Kunz multiplicity
as a function on the spectrum and in a one-parameter family. 
In both settings, they established weaker forms of semicontinuity  
\cite[Theorem~2.5, Theorem~2.6]{EnescuShimomoto}.
For the spectrum a complete solution was obtained 
by the author in \cite{semi, equi}, and the goal of this article  is to establish semicontinuity for a class of families 
similar to the situation in Example~\ref{ex Monsky} (see Definition~\ref{def affine family}).

Our definition of a family is versatile 
enough to include another outstanding  problem in the field: the behavior of 
Hilbert--Kunz multiplicity in reduction mod $p$.
For an illustration, consider the family
$\mathbb Z \to R:=\mathbb Z[x,y,z]/(z^4 +xyz^2 + (x^3 + y^3)z + x^2y^2)$.
A natural way to define the Hilbert--Kunz multiplicity of the generic fiber  
$\mathbb Q[x,y,z]/(z^4 +xyz^2 + (x^3 + y^3)z + x^2y^2)$, a ring of characteristic zero,
would be by taking the limit of Hilbert--Kunz multiplicities of special fibers 
$\lim_{p \to \infty} \ehk(R(p))$, and the question is whether the limit exists.

Hilbert--Kunz multiplicity is independent of characteristic for several classes of ``combinatorial'' rings because it only depends on the combinatorial data, for example: Stanley--Reisner rings (\cite{Conca}),
toric rings (\cite{Watanabe}), monoid algebras \cite{Eto, Bruns}, and 
binoid algebras, generalizing the previous cases (\cite{BrennerBatsukh}). 
Monsky's work provides examples where Hilbert--Kunz multiplicity depends 
on the characteristic (\cite{GesselMonsky}), but the only general case where this problem was solved is for graded rings of dimension two \cite{Trivedi, BrennerLiMiller}.
In an attempt to simplify the problem, in \cite{BrennerLiMiller} 
Claudia Miller asked
whether it is possible to replace the double limit $\lim_{p \to \infty} \ehk(R(p))$  by a single limit of the individual terms 
$\lim_{p \to \infty} \length (R(p)/\frq{\mf m}R(p))/p^{ed}$ for a fixed $e \geq 1$.
A positive answer to this question (and a more general statement)
was recently announced  by P{\' e}rez, Tucker, and Yao (\cite{PerezTuckerYao}). 
The methods of this paper provide 
an easy proof of this result in a special case (Corollary~\ref{Miller q}) and generalize 
a recent result of Trivedi (\cite{Trivedi2}) which was established 
in the graded case. However, neither this paper nor \cite{PerezTuckerYao} provide 
new cases in which $\lim_{p \to \infty} \ehk(R(p))$ is known to exist,
but rather make a step in Miller's approach. 

Another application of this work is in the theory of F-rational signature, an invariant introduced by Hochster and Yao in \cite{HochsterYao}. If $(R, \mf m)$ is a local ring, then its F-rational signature is defined by
\[
\rsig(R) = \inf_{u} \left\{ \ehk(\ul{x}) - \ehk(\ul{x}, u)\right\}
\]
where the infimum is taken over socle elements $u$
modulo a system of parameters $\ul{x}$.
Proposition~\ref{prop rat sig} proves that
if the residue field is algebraically closed, then the infimum in the definition is attained. 
This gives a fundamentally different proof of the main property of F-rational signature (\cite[Theorem~4.1]{HochsterYao}): 
$\rsig(R)$ is positive if and only if $R$ is F-rational. 

Last, we want to mention 
that using results in \cite{PerezTuckerYao}
Carvajal-Rojas, Schwede, and Tucker
\cite{CarvajalSchwedeTucker} recently obtained results in the spirit of this work.
However, their motivation is to study the behavior of Hilbert--Kunz multiplicity
in a family of varieties, while this work focuses on a family of ideals which are not necessarily maximal.
 
\subsection*{The methods and the structure of the paper}

This paper uses the uniform convergence method that was introduced by Tucker in \cite{Tucker}
to show convergence of F-signature as a limit and was later extended by the author in \cite{semi} to
show semicontinuity of Hilbert--Kunz multiplicity.
Polstra and Tucker in \cite{PolstraTucker} gave 
a more ``functorial'' approach to the uniform convergence constants
based on the discriminant technique in tight closure theory (\cite[Section~6]{HochsterHuneke1}). This approach was then applied by Polstra and the author \cite{PolstraSmirnov} to study Hilbert--Kunz 
multiplicity under small perturbations. 
The uniform convergence machinery of this paper is largely a mix of the techniques developed in \cite{PolstraSmirnov} and \cite{semi}.  Moreover, the appearing constants can be made independent of the characteristic, which gives a uniform convergence statement for fibers even if the base ring has characteristic zero (Corollary~\ref{ucon}). 
It should be noted that \cite[Proposition~4.5]{CarvajalSchwedeTucker} can be used to get a version
of Theorem~\ref{convergence thm} under stronger assumptions. 

Section~\ref{s prelim} slightly expands on \cite{PolstraTucker}
by further incorporating ideas from \cite{HochsterHuneke1}. 
Section~\ref{s HS} presents old and new results 
on the behavior of Hilbert--Samuel function in families. 
Definition~\ref{def affine family} introduces the assumptions of this work. 
The main results are presented in Section~\ref{s main}
and we finish with questions coming from this work.

\section{Discriminants and separability}\label{s prelim}

\begin{definition}
Let $R$ be a ring and $A$ a finite $R$-algebra which is a free $R$-module. If $e_1, \ldots, e_n$ are a free basis of $A$, then the discriminant of $A$ over $R$ is defined as
\[
\disc_R(A) = \det \begin{pmatrix} 
\Trace (e_1^2) & \Trace (e_1 e_2) & \cdots & \Trace (e_1e_n) \\
\Trace (e_2e_1) & \Trace (e_2^2) & \cdots & \Trace (e_2e_n)\\
\vdots & \vdots & \cdots & \vdots \\
\Trace (e_ne_1) & \Trace (e_n e_2) & \cdots & \Trace (e_n^2)
\end{pmatrix},
\]
where $\Trace (A)$ denotes the trace of the multiplication map $\times A$ on $A$.
Up to multiplication by a unit of $R$, the discriminant is independent of the choice of basis.
Discriminants are also functorial in $R$, for example, see \cite{PolstraSmirnov}. 
\end{definition}


We start with a fundamental lemma provided by Hochster and Huneke in \cite[Lemmas~6.4, 6.5]{HochsterHuneke1}.

\begin{lemma}\label{HH disc}
Let $R$ be a normal domain of characteristic $p > 0$ and $A$ be a module-finite, torsion-free, and generically separable $R$-algebra. Let $L$ be the fraction field of $R$, $L' = A \otimes_R L$, and $d = \disc_L (L')$ computed using a basis of elements in $A$. Then $0 \neq d \in R$ and $d A^{1/p} \subseteq R^{1/p}[A] \cong R^{1/p} \otimes_R A$.
\end{lemma}

The lemma also provides a way to define a discriminant of a non-free algebra. 
We will abuse the notation and still denote it by $\disc_R (A)$.
If $A$ is not torsion-free, we will use
the ideal
$T_R(A) = \{a \in A \mid ar = 0 \text{ for some } 0 \neq r \in R\}$.

\begin{corollary}\label{cor sequences}
Let $R$ be a normal domain and $A$ be module-finite and generically separable $R$-algebra.
Let $L$ be the fraction field of $R$, $L' = A \otimes_R L$, and 
$d = \disc_R (A)$ computed as in Lemma~\ref{HH disc}.
If $c \in R$ such that $cT_R(A) = 0$, 
then we have maps $\alpha\colon R^{1/p} \otimes_R A \to F_* A$
and $\beta \colon F_* A \to R^{1/p} \otimes_R A$
such that $cd (\coker \alpha) = 0$ and $cd (\coker \beta) = 0$.
\end{corollary}
\begin{proof}
Multiplication by $c$ on $A$ induces a map $A' := A/T_R(A) \xrightarrow{\times c} A$.
Observe that $A'$ is still generically separable over $R$, 
since $L' = A \otimes_R L = A' \otimes_R L$.
Hence $d F_* A' \subseteq R^{1/p}[A'] \cong R^{1/p} \otimes_R A'$ by Lemma~\ref{HH disc}.

Now we construct the maps in the assertion as compositions:
\[
\alpha\colon R^{1/p} \otimes_R A \to R^{1/p} \otimes_R A' \to F_* A' \xrightarrow{\times F_* c} F_* A,
\]
where the first map is natural and the second map is the multiplication $F_*r \otimes a \mapsto F_* a^pr$, and
\[
\beta\colon F_* A \to F_* A' \xrightarrow{\times d} R^{1/p} \otimes_R A' 
\xrightarrow{1 \otimes \times c} R^{1/p} \otimes_R A.
\]
For the first map, we note that
$cd F_* A \subseteq d F_* A' \subseteq R^{1/p}[A']$ by Lemma~\ref{HH disc}. Because  $R^{1/p}[A']$ is the image of 
$\alpha$, $\coker \alpha$ is annihilated by $cd$.
In the second map, we note that $F_* A \to F_* A'$ is surjective, 
$R^{1/p} [A'] \subseteq F_* A'$, and $cA = cA'$, so it follows that the cokernel of $\beta$
is annihilated by $cd$.
\end{proof}

The corollary becomes especially powerful after combining it with 
another result of 
Hochster and Huneke \cite[Lemma~6.15]{HochsterHuneke1}. 

\begin{lemma}\label{HH sep}
Let $A$ be a reduced ring, module-finite over a regular ring $R$ of characteristic $p > 0$.
Then for all sufficiently large $e$, $A \otimes_R R^{1/p^e}$ is module-finite and generically separable over $R^{1/p^e}$.
\end{lemma}
\begin{proof}
Let $L$ be the fraction field of $R$ and  $L' = A \otimes_R L$. Since $A$ is reduced, $L'$ is a product of fields. 
Tensoring with $L$ we get that  
\[
A \otimes_R R^{1/p^e} \otimes_R L = (A \otimes_R L) \otimes_L (R^{1/p^e} \otimes_R L)
= L' \otimes_L L^{1/p^e}.
\]
Hence the statement is reduced to the field case.
\end{proof}

\begin{corollary}\label{HH seq}
Let $A$ be a reduced ring, module-finite over a regular ring $R$ of characteristic $p > 0$.
Let $c \in R$ such that there exists a free $R$-module $F\subseteq A$
such that $cA \subseteq F$.  
Then for large $e$ we have exact sequences of $A$-modules
\[
R^{1/p^{e + 1}} \otimes_R A \to F_* (A \otimes_R R^{1/p^e}) \to C_{1, e} \to 0
\]
and 
\[
F_* (A \otimes_R R^{1/p^e}) \to R^{1/p^{e + 1}} \otimes_R A \to C_{2,e} \to 0,
\]
where the cokernels are annihilated by $c\disc_{R^{1/p^e}} (A \otimes_R R^{1/p^e})$.
\end{corollary}
\begin{proof}
We take $e$ large enough to satisfy Lemma~\ref{HH sep}.
Let $A' = A \otimes_R R^{1/p^e}$, $R' = R^{1/p^e}$, and $F' = F \otimes_R R^{1/p^e}$.
Because $R^{1/p^e}$ is flat by \cite{Kunz1}, 
$c A' \subseteq F'$, so
$cT_{R'}(A') \subseteq cA' \subseteq F'$ and $cT_{R'}(A') = 0$
because $F'$ is torsion-free. Now, we may use Corollary~\ref{cor sequences}
for $A'$ and $R'$.
\end{proof}

\section{Families and semicontinuity}\label{s HS}

We adopt the following notion of a family from \cite{Lipman}. 
Let $R$ be a ring, $A$ be an $R$-algebra, and $I \subset A$ be an ideal such that $A/I$ is a finitely generated $R$-module. 
For any prime ideal $\mf p \in \Spec R$ define $A(\mf p) := A \otimes_R k(\mf p)$ and $I(\mf p) = IA(\mf p) := I(A(\mf p))$. By the assumption, $A(\mf p)/I(\mf p) = A/I \otimes_R k(\mf p)$ has finite length. 
Thus, $I(\mf p)$ is a family of finite colength ideals in a family of rings $A(\mf p)$  parametrized by $\Spec R$.
If $M$ is a finite $A$-module, then
$M(\mf p):= M \otimes_R k(\mf p)$ is a finite $A(\mf p)$-module for all $\mf p \in \Spec R$.

Hilbert--Kunz multiplicity (and Hilbert--Samuel multiplicity) is now 
a real-valued function on $\Spec R$ via $\mf p \mapsto \ehk(I(\mf p), A(\mf p))$.
An example of such function is given in Example~\ref{ex Monsky} by 
a family $K[t] \to K[x,y,z, t]/(z^4 +xyz^2 + (x^3 + y^3)z + tx^2y^2)$
with $I = (x,y,z)$.

We also fix the following definition of semicontinuity. 
\begin{definition}\label{def semi}
Let $X$ be a topological space and $(\Lambda, \prec)$ be a partially ordered set.
We say that a function $f \colon X \to \Lambda$
is {\it upper semicontinuous} if for each $\lambda \in \Lambda$
the set
\[
X_{\prec \lambda} = \{x \in X \mid f(x) \prec \lambda\}
\]
is open. 
\end{definition}

In the literature, one can find an alternative definition of semicontinuity that instead requires
the sets $X_{\preceq \lambda} = \{x \in X \mid f(x) \preceq \lambda\}$ to be open.
This definition is stronger than Definition~\ref{def semi} but coincides if $f$ is discretely valued.
As it was observed by Enescu and Shimomoto (\cite[Theorem~2.7]{EnescuShimomoto}),
Monsky's example shows that Hilbert--Kunz multiplicity
is not an upper semicontinuous function in this, stronger sense (take $\lambda =3$).

\begin{remark}\label{rem Nagata}
Nagata's criterion of openness (\cite[22.B]{Matsumura})
is often used to show that a function is semicontinuous.
Namely, if $R$ is Noetherian, then a function $f \colon \Spec R \to \Lambda$
is upper semicontinuous if and only if the following two conditions hold:
\begin{enumerate}
\item if $\mf p \subset \mf q$ then $f(\mf p) \preceq f(\mf q)$, 
\item if $f(\mf p) \prec \lambda$ then there exists an elements 
$s \notin \mf p$ such that for every $\mf q$ with $s \notin \mf q$ and $\mf p \subseteq \mf q$ 
we have $f(\mf q) \prec \lambda$.
\end{enumerate}
\end{remark}

\subsection{Hilbert--Samuel function in families}
The theory of families of ideals originates from the work of Teissier 
(\cite{Teissier}) on the principle of specialization of integral closure 
and was further developed by Lipman in \cite{Lipman}.

We start with a lemma found in the proof of \cite[Proposition~4.2]{FlennerManaresi}.

\begin{lemma}\label{iso}
Let $R \to A$ be a map of Noetherian rings and $I$ be an ideal of $A$ such that $R \to A/I$ is finite. 
Suppose $M$ is a finite $A$-module. 
If $\gr_I(M)$ is flat over $R$, then for every finite $R$-module $N$ the canonical map
\[
\gr_I(M) \otimes_R N \to \gr_I (M \otimes_R N)
\] 
is an $A$-isomorphism.
\end{lemma}
\begin{proof}
It is sufficient to show that  for all $n$ the natural map $I^nM \otimes_R N \to I^n(M \otimes_R N)$
is an $A$-isomorphism. 
Because $R$ acts on $M \otimes_R N$ by multiplication on $M$, the map is surjective, 
so it remains to check injectivity. 

Because $I^nM/I^{n+1}M$ is a flat $R$-module as a direct summand of $\gr_I (M)$, there is an exact sequence
\[
0 \to I^{n+1}M \otimes_R N \to I^nM \otimes_R N \to (I^nM/I^{n+1}M) \otimes_R N \to 0.
\]
Using induction on $n$ it is now easy to verify the natural maps $I^nM \otimes_R N \to I^n (M \otimes_R N)$
are injective. 
\end{proof}

Using this lemma we are able to expand \cite[Proposition~3.1]{Lipman}.

\begin{theorem}\label{t HS}
Let $R \to A$ be a map of Noetherian rings and $I$ be an ideal in $A$ such that $A/I$ is a finite $R$-module.
Let $M$ be a finitely generated $A$-module.
Then the following functions on $\Spec R$ are upper semicontinuous:
\begin{enumerate}
\item $\mf p \mapsto \dkp M(\mf p)/I^nM(\mf p)$ for any $n$,
\item $\mf p \mapsto \left (\dkp  M(\mf p)/IM(\mf p), \dkp M(\mf p)/I^2M(\mf p), \ldots \right)$ (with lex-order),
\end{enumerate}

\end{theorem}
\begin{proof}
It can be shown by induction that, for all $n$, the modules
$M/I^nM$ and $I^nM/I^{n+1}M$ are finitely generated $R$-modules.
Observe that $M(\mf p)/I^nM(\mf p) \cong R/I^n \otimes_R M(\mf p) \cong M/I^nM \otimes_R k(\mf p)$.
But for any finite $R$-module $N$, $\dim_{k(\mf p)} N \otimes_R k(\mf p)$ is the minimal 
number of generators of $N(\mf p)$, which is clearly an upper semicontinuous function, see for example \cite[Lemma~2.2]{PolstraTucker}.
In particular, we obtain that the first condition of Nagata's criterion from
Remark~\ref{rem Nagata} is satisfied.

For the second condition, we provide a neighborhood of $\mf p$ where the functions are constant.
Observe that $\gr_I(M)$ is a finitely generated module over a finitely generated $R$-algebra, because it is a finite $\gr_I(A)$-module and $\gr_I(A)$ is a finitely generated module over $A/I[x_1, \ldots, x_N]$ where $x_1, \ldots, x_N$ are homogeneous generators of $\gr_I(A)$ of degree one.
For a fixed prime ideal $\mf p \in \Spec R$, we may apply generic freeness (\cite[22.A]{Matsumura}) over $R/\mf p$ for the module $\gr_I (M/\mf pM)$. 

In the resulting neighborhood $D(s)$ where $\gr_I(M/\mf pM)$ is free, by Lemma~\ref{iso} 
and flatness of localization, 
for all $\mf q \in D(s) \cap V(\mf p)$ we have the isomorphism
\[
\gr_I(M/\mf pM) \otimes_R k(\mf q) \cong \gr_I (M \otimes_R k(\mf q)).
\]
Because each $(I^n + \mf p)M/(I^{n+1} + \mf p)M$ is projective, it follows that 
$\dkq I^nM(\mf q)/I^{n+1}M(\mf q)$ is constant on $V(\mf p) \cap D(s)$ for all $n$.
\end{proof}

\begin{corollary}[{\cite[Proposition~3.1]{Lipman}}]
Let $R \to A$ be a map of Noetherian rings and $I \subset A$ be an ideal such that $A/I$ is a finite $R$-module.
If $\mf p \subseteq \mf q \subset R$ are prime ideals 
and $M$ is a finitely generated $A$-module, then 
$\dim M(\mf p) \leq  \dim M(\mf q)$ 
and if $\dim M(\mf p) = \dim M(\mf q)$ then 
$\eh(IM(\mf p)) \leq \eh(IM(\mf q))$. 
\end{corollary}

\begin{corollary}\label{HS bound}
Let $R \to A$ be a map of Noetherian rings and $I$ be an ideal such 
that $A/I$ is a finite $R$-module. 
Let $d = \max_{\mf m \in \Max R} \dim A/\mf mA$.
Then there exists a constant $C$ such that
for all $\mf p \in \Spec R$ and all $n$
\[
\dkp A(\mf p)/I^n(\mf p) < C n^d.
\]
\end{corollary}
\begin{proof}
First, note that if $\mf p \subseteq \mf m$ then
$\dim A/\mf mA = \dim A(\mf m) \geq \dim A(\mf p).$
So, for every $\mf p$, there is some constant $C(\mf p)$ that will work for all $n$. 
Given any $C$ the set 
\[
U(C) = \{\mf p \mid \dkp A(\mf p)/I^n(\mf p) < C n^d \text{ for all }n\}
= \cap_{n = 1}^\infty \{\mf p \mid \dkp A(\mf p)/I^n(\mf p) < C n^d\}
\]
is open by Theorem~\ref{t HS}. Thus we can build $C$ by Noetherian induction:
we first choose $C$ to be the maximum $C(\mf p)$ over the generic points
and then keep increasing $C$ by considering generic points of the complement 
of $U(C)$ until $U(C) = \Spec R$. 
\end{proof}

The following result of Lipman (\cite[Proposition~3.3]{Lipman})
provides a natural sufficient condition for equidimensionality of a family. 

\begin{lemma}\label{equi}
Let $R \to A$ be a map of Noetherian rings and $I$ an ideal of $A$ such that $A/I$ is a finite $R$-module
and $R \cap I = 0$.
Furthermore, assume that
\begin{enumerate}
\item $\hght \mf q + \dim A/\mf q = \dim A$ for every prime ideal $\mf q \supseteq I$ in $A$,
\item $\dim A/\mf mA + \dim R = \dim A$ for every maximal ideal $\mf m$ of $R$.
\end{enumerate} 
Then for every prime ideal $\mf p$ of $R$ we have $\dim A(\mf p) = \dim A - \dim R = \hght I$.
\end{lemma}

Due to the fundamental nature of Lemma~\ref{equi}, we would like 
to call the map $R \to A$ satisfying its assumptions an $I$-family.

\begin{definition}\label{def affine family}
We say that $R \to A$ is an {\it affine $I$-family} if $A$ is a finitely generated $R$-algebra and 
$I \subset A$ is an ideal such that
\begin{enumerate}
\item $A/I$ is a finite $R$-module,
\item $R \cap I = 0$,
\item $\hght \mf q + \dim A/\mf q = \dim A$ for every prime ideal $\mf q \supseteq I$ in $A$,
\item $\dim A/\mf mA + \dim R = \dim A$ for every maximal ideal $\mf m$ of $R$.
\end{enumerate}
\end{definition}
The second condition guarantees that $I(\mf p) \neq A(\mf p)$ for every $\mf p$. 
We can always pass to such a family by factoring by $I \cap R$. 
If $A$ is formally equidimensional then it satisfies (3),
if $A$ is a flat $R$-algebra, then it satisfies (4).
In particular, Example~\ref{ex Monsky} is 
coming from an affine $(x,y,z)$-family: localization does not change  the Hilbert--Kunz multiplicity because the Frobenius powers are $(x,y,z)$-primary. 

\section{Semicontinuity}\label{s main}

We want to show that $\ehk(I(\mf p))$ is an upper semicontinuous function on $\Spec R$ in the 
sense of Definition~\ref{def semi}. In order to build the uniform convergence machinery, we start with auxiliary lemmas.

\begin{lemma}\label{individual}
Let $R \to A$ be a map of rings of characteristic $p > 0$ and $I$ be an ideal in $A$ such that $A/I$ is a finite $R$-module.
For each integer $e \geq 1$ the function $\mf p \to \dim_{k(\mf p)} (A(\mf p)/I(\mf p)^{[p^e]})$
is upper semicontinuous on $\Spec R$.
\end{lemma}
\begin{proof}
If $I$ can be generated by $h$ elements, then $I^{hp^e} \subseteq \frq{I}$, 
so $A/\frq{I}$ is a finite $R$-module as in Theorem~\ref{t HS}. 
Thus $\mf p \to \dim_{k(\mf p)} (A(\mf p)/I(\mf p)^{[p^e]})$ is the minimal number of generators 
of that module at $\mf p$ and is an upper semicontinuous function.
\end{proof}

\begin{corollary}
Let $R \to A$ be an $I$-family as in Lemma~\ref{equi}. 
Then  for every $\mf p \subseteq \mf q$ we have $\ehk(I(\mf p)) \leq \ehk(I(\mf q))$.
\end{corollary}
\begin{proof}
Observe that $\dim A(\mf p)  = \hght I$ by Lemma~\ref{equi}, so we may pass to the limit in Lemma~\ref{individual}.
\end{proof}

\begin{lemma}\label{int closed}
Let $R$ be a Noetherian ring and let $A$ be an intersection flat $R$-algebra, {\it i.e.}, $\cap_{\lambda \in \Lambda} (I_\lambda A) = (\cap_{\lambda \in \Lambda} I_\lambda) A$ 
for arbitrary $\Lambda$ and ideals $I_\lambda \subset R$.
Then for any element $f \in A$ the set
\[
V_R (f) := \{\mf p \in \Spec R \mid f \in \mf pA\}
\]
is closed.
\end{lemma}
\begin{proof}
Let $I$ be the intersection of all primes in $V_R(f)$. 
Then $f \in \cap_{\mf p \in V_R(f)} \mf pA = (\cap_{\mf p \in V_R(f)} \mf p)A = IA$.
Hence $V_R(f) = V(I)$. 
\end{proof}

Last, we record a crucial lemma that provides a uniform upper bound for the main result. Note that polynomial extensions
are intersection flat. 

\begin{lemma}\label{HK bound}
Let $R$ be a Noetherian domain, $A = R[T_1, \ldots, T_d]$, and $I$
be an $(T_1, \ldots, T_d)$-primary ideal.
Let $M$ be a finite $A$-module annihilated by $0 \neq f \in A$.
Then there exists a constant $D$ with the following property: for any $e \geq 0$ and 
$\mf p$ in the open subset $\Spec R \setminus V_R(f)$ with $p := \chr k(\mf p)$
we have
\[
\dkp M(\mf p)/\frq{I}M(\mf p) < D p^{e (d - 1)},
\]
where the characteristic of $k(\mf p)$ may depend on $\mf p$.
\end{lemma}
\begin{proof}
For every maximal ideal $\mf m \notin V_R(f)$ 
\[
\dim A/(f, \mf m)A = \dim R/\mf m [T_1, \ldots, T_d]/(f) \leq d - 1.
\]
Let $N$ be such that $(T_1, \ldots, T_d)^N \subseteq I$.
Then we have inclusions
\[
(T_1, \ldots, T_d)^{Ndp^e} \subseteq ((T_1, \ldots, T_d)^{[p^e]})^N \subseteq \frq{I}.
\]
Suppose that $M$ can be (globally) generated by $\nu$ elements. 
We note that $\Spec R \setminus V_R(f)$ is a finite 
union of principal open set $D(c)$ and
for each $c$ we may apply 
Corollary~\ref{HS bound} to the map $R_c \to A_c$ and
estimate
\begin{align*}
\dkp M(\mf p)/\frq{I}M(\mf p) &\leq \nu \dkp A(\mf p)/\frq{I(\mf p)} 
< \nu C (Nd p^e)^{d - 1}
= (\nu C N^{d-1}d^{d - 1}) p^{e (d - 1)}.
\end{align*}
\end{proof}

\subsection{Main result}

Before proceeding to the proof of the main theorem we recall two lemmas. 
The first is due to Kunz \cite{Kunz2}.

\begin{lemma}\label{Kunz}
Let $R$ be a Noetherian ring of characteristic $p > 0$. 
Then for every $\mf p \subset \mf q$ 
\[
[k(\mf q)^{1/p^e} : k (\mf q)] = p^{e\dim R_\mf q/\mf p}[k(\mf p)^{1/p^e} : k (\mf p)].
\]
\end{lemma}

Second, we will need the following form of the Noether normalization theorem from \cite[14.4]{Nagata}.

\begin{theorem}\label{Noether}
Let $R$ be a domain and $A$ be a finitely generated $R$-algebra.
Then there exists an element $0 \neq c \in R$ such that 
$A_c$ is module-finite over a polynomial subring $R_c[z_1, \ldots, z_d]$.
\end{theorem}

\begin{theorem}\label{convergence thm}
Let $R$ be a regular F-finite ring of characteristic $p > 0$ and 
$R \to A$ be an affine $I$-family with reduced fibers of dimension $h = \hght I$. 
Then there exists 
an open set $U \subseteq \Spec R$ and a constant $D$ 
such that for all $\mf q \in U$  and all $e \geq 1$
\[
\left |\frac{\dkq A(\mf q)/I^{[p^{e+1}]}A(\mf q)}{p^{(e +1)h}}
- 
\frac{\dkq A(\mf q)/I^{[p^{e}]}A(\mf q)}{p^{eh}}\right |
< \frac{D}{p^{e}}.
\]
\end{theorem}
\begin{proof}
Because $A(0)$ is reduced, after inverting an element of $R$ 
we may assume that $A$ is reduced. Next, by Theorem~\ref{Noether} we invert another element and assume that $A$ is finite over $S = R[T_1, \ldots, T_h]$. 

Applying Lemma~\ref{HH sep} to the pair $S \subseteq A$ we find $e_0$ such that
$S^{1/p^{e_0}} \to A \otimes_S S^{1/p^{e_0}}$ is generically separable.
Since $S$ is a domain, there exists a free module $F \subseteq A$ 
and an element $0 \neq c \in S$ such that $cA \subseteq F$.
Because $S^{1/p^{e_0}}$ is flat, 
$F \otimes_S S^{1/p^{e_0}} \subseteq A \otimes_S S^{1/p^{e_0}}$ 
is a free submodule and $c$ still annihilates the cokernel.
Let $d^{1/p^{e_0}}$ to be the discriminant of 
$A \otimes_S S^{1/p^{e_0}}$ over $S^{1/p^{e_0}}$. 

\begin{claim}
Let $\mf q$ be a prime ideal in the open set $\Spec R \setminus V_R(cd)$.
Then $F(\mf q)$ is a free submodule of $A(\mf q)$ such that $cA(\mf q) \subseteq F(\mf q)$.
\end{claim}
\begin{proof}[Proof of the claim]
We have the induced map $F \otimes_S S(\mf q) \to A \otimes_S S(\mf q)$
whose cokernel is annihilated by the image of $c$ in $S(\mf q)$.
The image of $c$ is nonzero by the assumption, $F_c \cong A_c$, and $c \notin \mf q S$, so $F(\mf q)$ and $A(\mf q)$
are still generically isomorphic as $S(\mf q)$-modules.
Thus, since $F(\mf q)$ is a free $S(\mf q)$-module and 
$S(\mf q) \cong k(\mf q)[T_1, \ldots, T_h]$ is a domain,
the induced map $F \otimes_S S(\mf q) \to A \otimes_S S(\mf q)$ is still an inclusion. 
\end{proof} 

By the functoriality of discriminants (as in \cite[Proposition~2.2]{PolstraSmirnov}), the image of $d$ is still a discriminant of 
$A(\mf q) \otimes_{S(\mf q)} S(\mf q)^{1/p^{e_0}}$ over $S(\mf q)^{1/p^{e_0}}$. 
Since $d \notin \mf q S$, the inclusion is still generically separable. 
Hence, by Lemma~\ref{HH seq}, we have sequences 
\begin{equation}\label{first}
A(\mf q) \otimes_{S(\mf q)} S(\mf q)^{1/p^{e_0 + 1}}  
\to F_* \left (A(\mf q) \otimes_{S(\mf q)} S(\mf q)^{1/p^{e_0}} \right) \to C_1 \to 0
\end{equation}
and 
\begin{equation}\label{second}
F_* \left (A(\mf q) \otimes_{S(\mf q)} S(\mf q)^{1/p^{e_0}} \right ) \to A(\mf q) \otimes_{S(\mf q)} S(\mf q)^{1/p^{e_0 + 1}} \to C_2 \to 0,
\end{equation}
where $cd C_1 = 0 = cd C_2$.
Tensoring these exact sequences with $A/\frq{I}$, we obtain that
\begin{align}\label{bounds}
|\dkq A(\mf q)/\frq{I}A(\mf q) \otimes_{S(\mf q)} S(\mf q)^{1/p^{e_0 + 1}}  
- \dkq A/\frq{I} \otimes_A F_* \left (A(\mf q) \otimes_{S(\mf q)} S(\mf q)^{1/p^{e_0}} \right)|
\\ \notag  \leq \max \left (\dkq C_1/\frq{I}C_1, \dkq C_2/\frq{I}C_2 \right).
\end{align}

\begin{claim}\label{dim claim}
Denote $\alpha (k(\mf q)^e) := p^{eh} [k(\mf q)^{1/e} : k(\mf q)]$.
There is a constant $D$ independent of $\mf q$ such that 
\[
\dkq C_1/\frq{I}C_1, \dkq C_2/\frq{I}C_2
< D p^{e(h-1)} \alpha (k(\mf q)^{p^{e_0 + 1}}).
\]
\end{claim}
\begin{proof}
The exact sequence (\ref{first}) induces a natural surjection on $C_1/\frq{I}C_1$ from
\[
F_* \left (A(\mf q) \otimes_{S(\mf q)} S(\mf q)^{1/p^{e_0}} \right) 
\otimes_A A/(cd, I^{[p^e]})
\cong F_*\left (A(\mf q)/(c^pd^p, I^{[p^{e + 1}]})A(\mf q) \otimes_{S(\mf q)} S(\mf q)^{1/p^{e_0}} \right).
\]
Since $S(\mf q)$ is a polynomial ring of dimension $h$, by
Lemma~\ref{Kunz}
$S(\mf q)^{1/p^{e_0}}$ is a free $S(\mf q)$-module of rank 
$\alpha(k(\mf q)^{p^{e_0}})$.
Then we may bound
\begin{align*}
\dkq C_1/\frq{I}C_1 \leq &\alpha(k(\mf q)^{p^{e_0}}) \dkq F_* (A(\mf q)/(c^pd^p, I^{[p^{e + 1}]}) A(\mf q)) 
=\\& \alpha(k(\mf q)^{p^{e_0}}) [k(\mf q)^{1/p}:k(\mf q)] \dkq A(\mf q)/(c^pd^p, I^{[p^{e + 1}]}) A(\mf q)
=\\& \alpha(k(\mf q)^{p^{e_0 + 1}}) p^{-h} \dkq A(\mf q)/(c^pd^p, I^{[p^{e+1}]})A(\mf q).
\end{align*}
Because $A/I$ is a finite $R$-module, $I$ is $(T_1, \ldots, T_h)$-primary,
so by Corollary~\ref{HK bound} applied to $A/(cd)$ 
we may find a constant $D$
independent of $\mf q$ such that
\[
\dkq (A(\mf q)/(c^pd^p, I^{[p^{e+1}]})A(\mf q))
\leq p \dkq (A(\mf q)/(cd, I^{[p^{e+1}]})A(\mf q))
\leq pDp^{(e+1)(h-1)} = D p^h p^{e(h-1)}, 
\]
thus
$
\dkq (C_1/\frq{I}C_1) < 
D  p^{e(h-1)} \alpha (k(\mf q)^{p^{e_0 + 1}}) .
$

The second bound is similar:
$C_2/\frq{I}$ is an image of $A(\mf q)/(cd, \frq{I})A(\mf q) \otimes_{S(\mf q)} S(\mf q)^{1/p^{e_0 + 1}}$, thus
\[
\dkq (C_2/\frq{I}C_2) \leq \alpha (k(\mf q)^{p^{e_0 + 1}}) \dkq A(\mf q)/(cd, \frq{I})A(\mf q) <
D p^{e(h-1)} \alpha (k(\mf q)^{p^{e_0 + 1}}).
\]
\end{proof}

As in the proof Claim~\ref{dim claim}, we also have
\[
\dkq A(\mf q)/\frq{I}A(\mf q) \otimes_{S(\mf q)} S(\mf q)^{1/p^{e_0 + 1}}  
= \alpha (k(\mf q)^{p^{e_0 + 1}}) 
\dkq A(\mf q)/I^{[p^{e}]}A(\mf q)
\]
and 
\[
\dkq A/\frq{I} \otimes_A F_* \left (A(\mf q) \otimes_{S(\mf q)} S(\mf q)^{1/p^{e_0}} \right)
= p^{-h} \alpha (k(\mf q)^{p^{e_0 + 1}}) 
\dkq A(\mf q)/I^{[p^{e+1}]}A(\mf q).
\]
Now, dividing (\ref{bounds}) by 
$p^{eh} \alpha(k(\mf q)^{p^{e_0 + 1}})$,
from Claim~\ref{dim claim} we obtain that
\[
\left |\frac{\dkq A(\mf q)/I^{[p^{e+1}]}A(\mf q)}{p^{(e +1)h}}
- 
\frac{\dkq A(\mf q)/I^{[p^{e}]}A(\mf q)}{p^{eh}}\right |
< \frac{Dp^{e(h-1)}}{p^{eh}} \leq \frac{D}{p^{e}}.
\]

\end{proof}

\subsection{Families over $\mathbb Z$}
A careful analysis of the proof shows that it can be applied even when the characteristic varies in a family.

\begin{theorem}\label{convergence theorem 2}
Let $R$ be a regular ring of characteristic $0$
and $R \to A$ be an affine $I$-family
with reduced fibers of dimension $h$.
Suppose that for every $\mf p \in \Spec R$
the residue field
$k(\mf p)$ is F-finite whenever it has positive 
characteristic. 
Then there exists 
an open set $U \subseteq \Spec R$ and a constant $D$  
with the following property:
if $\mf q \in U$ and 
$p :=\chr k(\mf q) > 0$
then
\[
\left |\frac{\dkq A(\mf q)/I^{[p^{e+1}]}A(\mf q)}{p^{(e +1)h}}
- 
\frac{\dkq A(\mf q)/I^{[p^{e}]}A(\mf q)}{p^{eh}}\right |
< \frac{D}{p^{e}}.
\]
Note that $p$, the characteristic of $k(\mf q)$, may vary in the family and $D$ 
is independent of $p$. 
\end{theorem}
\begin{proof}
After inverting an element if necessary, 
we choose a Noether normalization $S = R_f[x_1, \ldots, x_d]$ of $A_f$.
Note that $S \subseteq A$ is generically separable, because $A(0)$ has characteristic $0$. 
So, we may proceed with the proof of Theorem~\ref{convergence thm} with $e_0 = 0$.
The constant $D$ in claim 
Claim~\ref{dim claim} comes from Lemma~\ref{HK bound} and does not depend on characteristic as it arises from the Hilbert--Samuel theory. 
\end{proof}

\begin{corollary}\label{ucon}
Let $R \to A$ be an affine $I$-family with reduced fibers of dimension $h$.
Suppose that for every $\mf p \in \Spec R$
the residue field $k(\mf p)$ is F-finite whenever it has positive 
characteristic ({\it e.g.,} $R$ is F-finite or $R = \mathbb Z$). 
Then there exists 
an open set $U \subseteq \Spec R$ and a constant $D$  
with the following property:
if $\mf q \in U$ and 
$p :=\chr k(\mf q) > 0$
then
\[
\left |\ehk(I(\mf p)) - 
\frac{\dkq A(\mf q)/I^{[p^{e}]}A(\mf q)}{p^{eh}}\right |
< \frac{2D}{p^{e}}.
\]
\end{corollary}
\begin{proof}
We may pass to $R/\mf p \to A/\mf p$ and assume that $\mf p =0$.
An F-finite ring is excellent (\cite[Theorem~2.5]{Kunz2}),
so the regular locus of $R$ is open and, by  inverting an element, we assume that $R$ is regular. 

Let $D$ be the constant provided by Theorem~\ref{convergence thm}
or Theorem~\ref{convergence theorem 2},
then the claim follows from the proof of \cite[Lemma~3.5]{PolstraTucker}.
\end{proof}

\begin{corollary}\label{cor semi}
Let $R$ be an F-finite ring of characteristic $p > 0$ and
$R \to A$ be an affine $I$-family with reduced fibers. 
Then the function $\mf p \mapsto \ehk (I(\mf p))$ 
is upper semicontinuous on $\Spec R$.
\end{corollary}
\begin{proof}
We use uniform convergence to pass semicontinuity from the individual term to the limit as in \cite{PolstraTucker, semi}.
Each individual term, $\dkp A(\mf p)/\frq{I}A(\mf p)$ is the number of generators of $A/\frq{I}$ at $\mf p$ and, thus,
is naturally upper semicontinuous.
\end{proof}

We have the following geometric consequence. 

\begin{corollary}
Let $X \to T$ together with a section $\sigma \colon T \to X$
be a flat family of finite type with reduced fibers over a variety $T$ of characteristic $p > 0$. 
Then the function $t \mapsto \ehk (\sigma(t), X_t)$ 
is upper semicontinuous on $T$.
\end{corollary}

The following corollary provides a positive answer 
to the question of Claudia Miller from \cite{BrennerLiMiller} and recovers the main 
result, \cite[Corollary~3.3]{BrennerLiMiller}.
A much more general result about reductions mod $p$ was announced in \cite{PerezTuckerYao}. 
For a family of geometrically integral graded rings and $e \geq h - 1$
an affirmative answer was recently obtained by Trivedi in \cite[Corollary~1.2]{Trivedi2}.

\begin{corollary}\label{Miller q}
Let $\mathbb Z \to R$ be an affine $I$-family 
with reduced fibers of dimension $h$. 
Then for every $e \geq 1$ 
\[
\lim_{p \to \infty} \left(\ehk(IR(p)) - \frac{\length(R(p)/I(p)^{[p^e]})}{p^{eh}} \right) = 0.
\]
\end{corollary}
\begin{proof}
By Corollary~\ref{ucon}, we obtain that for all 
sufficiently large $p$
\[
\left |\ehk(IR(p)) - \frac{\length(R(p)/I(p)^{[p^e]})}{p^{eh}} \right| < \frac{2D}{p^{e}}.
\]
and the theorem follows.
\end{proof}

\subsection{F-rational signature}
In \cite{HochsterYao} 
Hochster and Yao introduced the following definition.

\begin{definition}
Let $(R, \mf m)$ be a local ring. 
The {\it F-rational signature} of $R$
is defined as 
\[
\rsig(R) = \inf_{u} \left\{ \ehk(\ul{x}) - \ehk(\ul{x}, u)\right\}
\]
where the infimum is taken over all systems
of parameters $\ul{x}$ and socle elements $u$.
\end{definition}

In \cite[Theorem~2.5]{HochsterYao}, it was shown that one can fix an arbitrary $\ul{x}$ in the definition.

\begin{proposition}\label{prop rat sig}
Let $k$ be a field of characteristic $p>0$, $R$ be a finitely generated $k$-algebra, and $\mf m$ be a maximal ideal of $R$.
Then the infimum in the definition of $\rsig(R_{\mf m})$ is achieved.
\end{proposition}
\begin{proof}
First of all, we observe that the assertion is equivalent to showing that 
$\ehk(\ul{x}, u)$ has a maximum where $\ul{x}$ is a system of parameter and $u$ varies through socle elements.

Let $u_1, \ldots, u_N$ be a basis of $(\ul{x}):\mf m/(\ul{x})$ as a $k$-vector space. 
We may parametrize the socle ideals $(I, u)$ via two affine families: $(\ul{x}, T_1u_1 + \cdots + T_{N-1}u_{N-1} + u_N)$-family
\[
S := k[T_1, \ldots, T_{N-1}] \to
R[T_1, \ldots, T_{N-1}]
\]
and, similarly, for $u_1 + T_2u_2 + \cdots + T_Nu_N$.
By Corollary~\ref{cor semi}, the function 
$f\colon \mf p \mapsto 
\ehk((\ul{x}, T_1u_1 + \cdots + T_{N-1}u_{N-1} + u_N)R(\mf p))$
is upper semicontinuous on $\Spec S$.

The claim now follows since an upper semicontinuous function satisfies an ascending chain condition.
Namely, let $u_n$ be a sequence of socle elements such that 
$\ehk(\ul{x}, u_n) < \ehk(\ul{x}, u_{n+1})$ for all $n$.
Without loss of generality we may assume that $u_n$ correspond to maximal ideals $\mf m_n$ of $S$.
Then $U_n = \{\mf p \in \Spec S \mid f(\mf p) < \ehk(\ul{x}, u_n) )\}$ is an increasing family of open sets
which cannot stabilize because $\mf m_n \in U_{n+1}\setminus U_n$. 
Since $\Spec S$ is Noetherian, this is a contradiction.
\end{proof}

\begin{remark}
We want to note that Proposition~\ref{prop rat sig}
can be also applied when $R$ is
given as a quotient of a power series
ring by an ideal generated by polynomials, since the lengths do not change under completion. 
By Artin's celebrated result (\cite[Theorem~(3.8)]{Artin})
this gives us that the conclusion holds for complete isolated singularities with a perfect residue field. 
\end{remark}

As a consequence, we recover 
a special case of 
\cite[Theorem~4.1]{HochsterYao}.

\begin{corollary}
Let $k$ be a field of characteristic $p>0$, $R$ be a finitely generated $k$-algebra, 
and $\mf m$ be a maximal ideal of $R$.
Then $\rsig(R_{\mf m}) > 0$ if and only if $R_{\mf m}$ is F-rational.
\end{corollary}
\begin{proof}
Recall that $R$ is F-rational if $\ul{x}$ is tightly closed or, equivalently, that $\ehk(\ul{x}) > \ehk(\ul{x}, u)$ for every socle element $u$.
\end{proof}

\begin{remark}
A variation of the Hochster--Yao definition, relative F-rational signature, was proposed in \cite{SmirnovTucker}
\[
\csig(R) = 
\inf_{\ul{x} \subset I} \frac{ \ehk(\ul{x}) - \ehk(I)}{\length (R/\ul{x}) - \length (R/I)},
\]
where the infimum is taken over all $\mf m$-primary ideals $I$ containing 
a system of parameters $\ul{x}$.
The paper shows that the definition also does not depend on the choice of $\ul{x}$
and that it might have better properties than $\rsig(R)$. 

By considering higher degree Grassmannians of $(\ul{x}):\mf m/(\ul{x})$, from the proof of  Proposition~\ref{prop rat sig} we may also get that the relative F-rational signature is a minimum. 
\end{remark}

\section{Questions}

\subsection{Nilpotents}
Like the preceding work \cite{PolstraSmirnov}, 
this paper has to assume that the family is reduced because
of the lack of control in non-reduced rings. 
While Hilbert--Kunz multiplicity exists for non-reduced rings, 
the original proof in \cite{Monsky} and its extensions pass
to $\red{R}$ by observing that $F_*^{e_0} R$ 
is an $\red{R}$-module for large $e_0$. 
This is not satisfactory for two reasons: 
the approach via discriminants does not adapt for modules
and we do not see how to control the exponent $e_0$. 

\subsection{F-signature}
F-signature is a measure of singularity in positive characteristic 
introduced by Huneke and Leuschke in \cite{HunekeLeuschke}.
Due to similarities between the two theories, it is natural to ask  whether the results of this paper extend to F-signature.

A related statement was observed in \cite[Theorem~4.9]{CarvajalSchwedeTucker},
however, it does not give lower semicontinuity since $A$ is assumed to be of finite type 
over a field and cannot be localized to apply Nagata's criterion. 
In fact, F-signature is not lower semicontinuous in families, because an example of Singh shows that 
strong F-regularity is not open (\cite{Singh}, also see \cite{DeStefaniSmirnov}).

\subsection{Localization of tight closure}
As it was mentioned above, in \cite{BrennerMonsky} Brenner and Monsky showed that tight closure does not localize. 
However, we do not understand the underlying reasons.
In particular, how does it relate to the results of
\cite{HochsterHuneke3} 
and how typical is this phenomenon?
As \cite{BrennerMonsky} depends on an irregular behavior 
of Hilbert--Kunz multiplicity in a family, we hope that it should be possible to give a 
general procedure for producing counter-examples 
from such families, for example, the family in \cite{MonskyFamily}. 
The study of Hilbert--Samuel multiplicity in families
was pioneered by Teissier (\cite{Teissier}) to give a criterion of equimultiplicity: $\eh(I(\mf p))$ is independent of $\mf p$
if and only if $\ell(I) = \hght (I)$.
The author suspects that a study of equimultiplicity in families
for Hilbert--Kunz multiplicity might explain the phenomenon presented in \cite{BrennerMonsky}.

\section*{Acknowledgements}
I thank Mel Hochster for helping to shape Corollary~\ref{cor sequences}, 
Javier Carvajal-Rojas and Thomas Polstra for discussions, 
Yongwei Yao for informing me about the results in \cite{PerezTuckerYao}, 
and  the anonymous referee for well-thought-out comments that greatly improved the quality of this paper.

\bibliographystyle{alpha}
\bibliography{bib}

\end{document}